\documentclass[leqno,11pt]{article}

\usepackage[utf8]{inputenc}
\usepackage[T1]{fontenc}
\usepackage{microtype}
\usepackage[a4paper]{geometry}
\ifdefined\screenview
  \edef\mtht{\the\textheight}
  \edef\mtwd{\the\textwidth}
\fi

\usepackage[dvipsnames]{xcolor}

\usepackage{amsmath}
\usepackage{amsthm}
\usepackage{tikz-cd}
\usetikzlibrary{arrows} 
\tikzset{
  commutative diagrams/.cd, 
  arrow style=tikz, 
  diagrams={>=stealth}
}
\usetikzlibrary{matrix,decorations.pathreplacing,calc}

\usepackage{textcomp}
\usepackage[sb]{libertine}
\usepackage[varqu,varl]{zi4}%
\usepackage[libertine,bigdelims,vvarbb]{newtxmath}
\useosf
\usepackage{csquotes}
\usepackage[
  backend=biber,
  hyperref=true,
  backref=true,
  isbn=false,
  doi=true,
  natbib=true,
  eprint=true,
  useprefix=true,
  maxcitenames=99,
  maxbibnames=99,  
  maxalphanames=99, 
  minalphanames=99,
  safeinputenc,
  style=alphabetic,
  citestyle=alphabetic,
  block=space,
  datamodel=preamble/ext-eprint,
  sorting=nyt
]{biblatex}
\usepackage[
  bookmarksnumbered = true,
  hypertexnames = false,
  linktocpage = true,
  colorlinks    = true,
  citecolor     = Green,
  linkcolor     = Blue,
  urlcolor      = Blue,
  breaklinks
]{hyperref}

\DeclareFieldFormat{url}{%
  \href{#1}{\ComputerMouse}
}
\DeclareFieldFormat{doi}{%
  \mkbibacro{DOI}\addcolon\space\href{https://doi.org/#1}{#1}
}
\makeatletter
\DeclareFieldFormat{arxiv}{%
  arXiv\addcolon\space\href{http://arxiv.org/\abx@arxivpath/#1}{#1}
}
\makeatother
\DeclareFieldFormat{mr}{%
  MR\addcolon\space\href{http://www.ams.org/mathscinet-getitem?mr=MR#1}{#1}
}
\DeclareFieldFormat{zbl}{%
  Zbl\addcolon\space\href{http://zbmath.org/?q=an:#1}{#1}
}
\renewbibmacro*{eprint}{%
  \printfield{arxiv}%
  \newunit\newblock
  \printfield{mr}%
  \newunit\newblock
  \printfield{zbl}%
  \newunit\newblock
  \iffieldundef{eprinttype}
  {\printfield{eprint}}
  {\printfield[eprint:\strfield{eprinttype}]{eprint}}
}

\AtEveryBibitem{%
  \clearlist{address}%
}
\DeclareFieldFormat[article,inproceedings,inbook,incollection,thesis]{title}{\textit{#1}}
\renewbibmacro{in:}{}
\usepackage[scr=boondox,cal=euler]{mathalfa}

\usepackage{marvosym}

\usepackage{slashed}
\usepackage{esint} 
\usepackage[english]{babel}

\usepackage{imakeidx}
\indexsetup{noclearpage}
\makeindex[intoc]

\usepackage[inline,shortlabels]{enumitem}

\usepackage{subcaption}

\usepackage[yyyymmdd]{datetime}

\usepackage{etoolbox}
\ifundef{\abstract}{}{\patchcmd{\abstract}%
    {\quotation}{\quotation\noindent\ignorespaces}{}{}}

\usepackage[super]{nth}

\usepackage{thmtools}

\numberwithin{equation}{section}
\renewcommand{\qedsymbol}{$\blacksquare$}

\newcommand{\CorollaryQED}{\qedsymbol}

\newcommand{\ConjectureQED}{$\square$}
\newcommand{\SituationQED}{$\times$}
\newcommand{\DefinitionQED}{$\spadesuit$}
\newcommand{\NotationQED}{$\blacktriangleright$}
\newcommand{\ExampleQED}{$\bullet$}
\newcommand{\RemarkQED}{$\clubsuit$}

\declaretheorem[numberlike=equation,]{theorem}
\declaretheorem[numbered=no,name=Theorem]{theorem*}
\declaretheorem[numberlike=equation,name=Lemma]{lemma}
\declaretheorem[numberlike=equation,name=Proposition]{prop}
\declaretheorem[numberlike=equation,name=Corollary,qed=\CorollaryQED]{cor}

\declaretheorem[numberlike=equation,name=Definition,style=definition,qed=\DefinitionQED]{definition}
\declaretheorem[numbered=no,name=Definition,style=definition,qed=\DefinitionQED]{definition*}

\declaretheorem[numberlike=equation,style=definition,qed=\ExampleQED]{example}

\declaretheorem[numberlike=equation,style=remark,qed=\RemarkQED]{remark}
\declaretheorem[numbered=no,style=remark,name=Remark,qed=\RemarkQED]{remark*}

\def\makeautorefname#1#2{\AtBeginDocument{\expandafter\def\csname#1autorefname\endcsname{#2}}}
\makeautorefname{table}{Table}        
\makeautorefname{chapter}{Chapter}
\makeautorefname{section}{Section}
\makeautorefname{subsection}{Section}
\makeautorefname{subsubsection}{Section}
\makeautorefname{footnote}{Footnote}
\AtBeginDocument{\def\itemautorefname~#1\null{(#1)\null}}
\AtBeginDocument{\def\equationautorefname~#1\null{(#1)\null}}

\numberwithin{substep}{step}
\makeautorefname{step}{Step}
\makeautorefname{substep}{Step}

\makeautorefname{case}{Case}
\makeautorefname{substep}{Step}

\setlist[description]{leftmargin=!,labelindent=1em}
\setlist[enumerate]{label={\rm (\arabic*)},ref=\arabic*}
\setlist[enumerate,2]{label={\rm (\alph*)},ref=\theenumi.\alph*}
\setlist[enumerate,3]{label={\rm (\roman*)},ref=\theenumii.\roman*}

\let\C\undefined
\let\U\undefined

\usepackage{bm}
\usepackage{mathtools} 
\usepackage{stmaryrd} 

\DeclareFontFamily{U}{mathx}{\hyphenchar\font45}
\DeclareFontShape{U}{mathx}{m}{n}{
     <5> <6> <7> <8> <9> <10>
      <10.95> <12> <14.4> <17.28> <20.74> <24.88>
      mathx10
     }{}
\DeclareSymbolFont{mathx}{U}{mathx}{m}{n}
\DeclareFontSubstitution{U}{mathx}{m}{n}
\DeclareMathAccent{\widecheck}{0}{mathx}{"71}
\DeclareMathAccent{\wideparen}{0}{mathx}{"75}

\DeclareMathOperator{\Diff}{Diff}
\DeclareMathOperator{\End}{End}

\DeclareMathOperator{\HF}{\HF}

\DeclareMathOperator{\sym}{Sym}

\DeclareMathOperator{\ad}{ad}

\DeclareMathOperator{\coker}{coker}

\DeclareMathOperator{\im}{im}

\DeclareMathOperator{\tr}{tr}

\DeclarePairedDelimiter{\norm}{\|}{\|}

\DeclarePairedDelimiterX{\inp}[2]{\langle}{\rangle}{#1, #2}

\DeclarePairedDelimiter{\abs}{\lvert}{\rvert}

\def\({\left(}
\def\){\right)}
\def\<{\left\langle}
\def\>{\right\rangle}

\newcommand{\C}{{\mathbf{C}}}

\newcommand{\PU}{{\P\U}}

\newcommand{\R}{\mathbf{R}}

\newcommand{\SO}{\mathrm{SO}}
\newcommand{\SU}{\mathrm{SU}}

\newcommand{\Spin}{\mathrm{Spin}}
\newcommand{\Sp}{\mathrm{Sp}}

\newcommand{\U}{\mathrm{U}}

\newcommand{\Z}{\mathbf{Z}}

\newcommand{\iso}{\cong}

\newcommand{\ob}{\mathrm{ob}}

\renewcommand{\epsilon}{\varepsilon}

\newcommand{\vol}{\mathrm{vol}}

\renewcommand{\H}{\mathbf{H}}
\renewcommand{\Im}{\operatorname{Im}}

\renewcommand{\P}{\mathbf{P}}

\renewcommand{\det}{\operatorname{det}}

\renewcommand{\leq}{\leqslant}
\renewcommand{\geq}{\geqslant}

\newcommand{\w}{\wedge}
\newcommand{\tn}{\otimes}

\makeatletter
\renewcommand*\env@matrix[1][*\c@MaxMatrixCols c]{%
  \hskip -\arraycolsep
  \let\@ifnextchar\new@ifnextchar
  \array{#1}}

\renewcommand\xleftrightarrow[2][]{%
  \ext@arrow 9999{\longleftrightarrowfill@}{#1}{#2}}
\newcommand\longleftrightarrowfill@{%
  \arrowfill@\leftarrow\relbar\rightarrow}
\makeatother

\newcommand{\bS}{{\mathbf{S}}}

\newcommand{\cL}{\mathcal{L}}

\newcommand{\sA}{\mathscr{A}}

\newcommand{\sG}{\mathscr{G}}

\newcommand{\sM}{\mathscr{M}}

\newcommand{\fs}{{\mathfrak s}}

\newcommand{\fM}{{\mathfrak M}}

\newcommand{\slD}{\slashed D}

\author{Gorapada Bera}
\title{Remarks on $\Sp(1)$-Seiberg-Witten equation over $3$-manifolds}

\begin{document}
\maketitle
\date
\begin{abstract}We prove that the $\Sp(1)$-Seiberg-Witten equation over a closed hyperbolic $3$-manifold ${\H}^3/\Gamma$ always admits a canonical irreducible solution induced by the hyperbolic metric. We also prove that the Zariski tangent space of the moduli space at this canonical solution is same as the Zariski tangent space of the moduli space of locally conformally flat structures at the hyperbolic metric. This space is again same as the space of trace-free Codazzi tensors and carries an injection to $H^1(\Gamma,\R^{1,3})$, the first group cohomology of the $\Gamma$-module $\R^{1,3}$. In particular, if $H^1(\Gamma,\R^{1,3})=0$ then the canonical irreducible solution is infinitesimally rigid. We also prove that the $\Sp(1)$-Seiberg-Witten equation over $S^1\times \Sigma$ has no irreducible solutions and the moduli space of reducible solutions is same as the moduli space of flat $\SU(2)$-connections.     
\end{abstract}
\section{Introduction: Main results}
One of the simplest generalized non-abelian Seiberg-Witten equation is the $\Sp(1)$-Seiberg-Witten equation. The representation of $\Sp(1)$ on the quaternion $\H$ given by the right multiplication after conjugation and a $\Spin^{h}$-structure $\fs$ on an oriented closed Riemannian $3$ or $4$-manifold $(M,g)$ defines this equation. \citet{Lim2003} has defined topological invariant of integral homology $3$-spheres by counting solutions (with some correction terms) of perturbed $\Sp(1)$-Seiberg-Witten equation. Although the tremendous success of the classical (abelian) Seiberg-Witten equation motivated people to study the non-abelian generalization $20$ years ago but has little success so far due to the presence of reducibles and non-compactness phenomena. Recently generalized Seiberg-Witten equation is again gaining attention due to its importance not only in low dimensional topology but also in special holonomy \cite{Doan2017d}.

In this article we will restrict ourselves to closed $3$-manifolds. It turns out that there is a unique $\Spin^{h}$ structure on a $3$-manifold up to isomorphism and the moduli space $\sM^h_{SW}$ of solutions of the $\Sp(1)$-Seiberg-Witten equation is compact. A solution is irreducible if and only the spinor is not identically zero. The main goal of this article is to address the following questions for certain $3$-manifolds:
\begin{enumerate}[1.]
\item Does there exist an irreducible solution?
\item Does there exist an irreducible infinitesimally rigid (unobstructed) solution?
\item Is it possible that no irreducible solutions exist?
\end{enumerate}
We answer the first two questions by choosing $M$ to be a closed hyperbolic $3$-manifold ${\H}^3/\Gamma$ and answer the third question by choosing $M$ to be $S^1\times \Sigma$, a product of circle with Riemann surface. We would like to point out that \citet{Walpuski18} first found the irreducible solution on hyperbolic $3$-manifold in an unpublished document.
We now state our main theorems.
\begin{theorem} \label{thm hyperbolic sol}Let $(M,g)$ be a hyperbolic $3$-manifold ${\H}^3/\Gamma$ with $\Gamma$ be a co-compact discrete subgroup of $\SO^+(1,3)$ and $g$ be the hyperbolic metric. Then the hyperbolic metric $g$ induces an irreducible solution $(A_0,\Phi_0)$ of the $\Sp(1)$-Seiberg-Witten equation \autoref{eq sp(1) SW main}. Moreover, the Zariski tangent space of the moduli space $\sM^h_{SW}$ of solutions at $(A_0,\Phi_0)$ is same as the Zariski tangent space of the moduli space $\fM_{lcf}$ of locally conformally flat structures on $M$ at $[g]$, and both are equal to the following space 
$$H^1(C,g):=\{h\in \sym^2_0(M,g) :d_{LC} h=0\}=\{\text{Trace-free Codazzi tensors on $(M,g)$}\}.$$ 
\end{theorem}
 \citet{Lafontaine1983} (see \autoref{thm Lafontaine}) had shown that there is an injective map $H^1(C,g)\hookrightarrow H^1(\Gamma,\R^{1,3})$, the first group cohomology of the $\Gamma$-module $\R^{1,3}$. Thus we obtain the following:
\begin{cor}
     If  the group cohomology $H^1(\Gamma,\R^{1,3})$ vanishes then the irreducible solution $(A_0,\Phi_0)$ in \autoref{thm hyperbolic sol} is infinitesimally rigid (unobstructed).
\end{cor}
It is known in the literature that there are infinitely many hyperbolic $3$-manifolds with $H^1(\Gamma,\R^{1,3})=0$ (see \autoref{eg rigid}). Also out of the first $4500$ two generator hyperbolic manifolds in the Hodgson-Weeks census, $4439$ has $H^1(\Gamma,\R^{1,3})=0$. For all those hyperbolic $3$-manifolds $(A_0,\Phi_0)$ is infinitesimally rigid.

\begin{theorem}\label{thm SW Riemann surface}Let $M$ be $S^1\times \Sigma$, a product of a circle with a Riemann surface $\Sigma$ with a product metric. Then the $\Sp(1)$-Seiberg-Witten equation \autoref{eq sp(1) SW main} over $M$ does not have any irreducible solution. In particular, the moduli space $\sM^h_{SW}$ can be identified with the moduli space of flat $\SU(2)$-connections over $M$.
\end{theorem}
\paragraph{Acknowledgements} I am grateful to my PhD supervisor Thomas Walpuski for introducing me to these equations and generously sharing the unpublished document \cite{Walpuski18}. Additionally, I extend my thanks to Misha Kapovich for answering my questions regarding examples of locally rigid hyperbolic metrics as a conformally flat structure.

 \section{$\Sp(1)$-Seiberg-Witten equation}In this section we discuss the basic set up and some identities for the $\Sp(1)$-Seiberg-Witten equation on an oriented Riemannian $3$–manifold. This will be an example of a generalized Seiberg-Witten equation discussed in \cite[Section 2]{Doan2017a}. 
Define quaternionic representations $\gamma:\Sp(1)\to \operatorname{End}(\H)$ by left multiplication and $\rho:\Sp(1)\to \operatorname{End}(\H)$ by right multiplication after conjugation that is,
 $$\gamma(p)\Phi=p\Phi,\ \ \ \ \rho(p)\Phi=\Phi\bar p.$$
 Denote their Lie algebra homomorphisms again by $\gamma, \rho:\Im \H \to \End(\H)$. Furthermore, we define $\tilde \gamma:$ Im $\H \otimes \Im \H \to \End(\H)$  by $$\tilde \gamma(v \otimes \xi)\Phi:=\gamma(v)\circ \rho(\xi)\Phi=-v\Phi\xi.$$
The map $\mu:\H\to(\Im \H\otimes \Im \H)^*$ defined by
	$$\mu(\Phi):=\frac 12\tilde \gamma^*(\Phi \Phi^*)$$
	is a \textbf{hyperk\"ahler moment map} that is, it is $\H$-equivariant and $\inp{(d\mu)_\Phi \phi}{v \otimes \xi}=\inp{\gamma(v)\rho(\xi)\Phi}{\phi}$.
	The corresponding bilinear map $\mu:\H\times \H\to(\Im \H\otimes \Im \H)^*$ is defined by
	$\mu(\Phi,\Psi):=\frac 12\tilde \gamma^*(\Phi \Psi^*)$.
	
Set $\Spin^{h}(3):= \frac{\Sp(1)\times \Sp(1)}{\<-1\>}\cong \SO(4)$. We have the following short exact sequence
\begin{equation}\label{eq exact 1}1\to \Sp(1)\xrightarrow{p\mapsto[(1,p)]}
 \Spin^{h}(3)\to \SO(3)\to 1.\end{equation}

\begin{definition}
	A \textbf{$\Spin^{h}$- structure} on an oriented closed Riemannian $3$-manifold $(M,g)$ is a principal $\Spin^{h}(3)$-bundle $\fs$ with an isomorphism  \begin{equation*}
  	\fs \times_{\Spin^{h}(3)} \SO(3)\cong \SO(TM).\qedhere   \end{equation*}
\end{definition}

The $\Spin$-structures and $\Spin^c$- structures on $(M,g)$ are all examples of $\Spin^{h}$- structures. But the following proposition says that all of them isomorphic. 
\begin{prop} An oriented closed Riemannian $3$-manifold $(M,g)$ always admits a $\Spin^{h}$- structure and it is unique up to isomorphism.
\end{prop}
\begin{proof} Since the $3$-manifold $(M,g)$ always admits a $\Spin$-structure (as $w_2(TM)=0$) we obtain the existence of a $\Spin^{h}$- structure. Now we prove the uniqueness. Two $\Spin^{h}$- structures $\fs_1,\fs_2$ are isomorphic if and only if the fiber bundle $\operatorname{Iso}_{\SO(TM)}(\fs_1,\fs_2)\to M$ has a section. Since the fibers ${\operatorname{Iso}_{\SO(TM)}(\fs_1,\fs_2)}_{|x})=\Sp(1)$, obstruction to the existence of such a section is an element in $H^4(M,\Z)=0$. Thus $\fs_1,\fs_2$ are isomorphic.
\end{proof}

A $\Spin^{h}$- structure $\fs$ on $(M,g)$ induces the following associated bundles and maps: 
\begin{itemize}
	\item The \textbf{spinor bundle}, $\bS:=\fs \times_{(\gamma\times \rho)} \H$,
		\item The \textbf{adjoint bundle}, $\ad(\fs):=\fs \times_{\rho} \Im \H$, 
	\item The \textbf{Clifford multiplication map} $\gamma: TM \to \End(\bS)$ induced by $\gamma:\Im \H\to \operatorname{End}(\H)$,
	\item $\rho:\ad(\fs) \to \End(\mathbf S),\ \ 
	\tilde \gamma: TM\otimes \ad(\fs) \to \End(\mathbf S)$, induced by $\rho$ and $\tilde \gamma$,  
	\item The \textbf{moment map} $\mu:\mathbf S \to \Lambda^2(T^*M)\otimes \ad(\fs)$, induced by the hyperk\"ahler moment map $\mu$.  
\end{itemize}
\begin{definition}
	A \textbf{spin connection} on a $\Spin^{h}$ structure $\fs$ is a connection on $\fs$ which induces the Levi-Civita connection on $TM$. Denote by $\sA(\fs)$ the space of all spin connections on $\fs$. Denote by $\ad(A)$ the connection on $\ad(\fs)$ induced by a spin connection $A$.
	We define the \textbf{group of gauge transformations} $\sG(\fs)$ by
		$$\sG(\mathfrak s):=\{u\in\text{Aut}(\mathfrak s): u \ \text{acts trivially on}\   \SO(TM) \}$$
		and the \textbf{action }of $\sG(\fs)$ on $\sA(\fs)\times \Gamma(\textbf{S})$ by
		$u\cdot(A,\Phi):=((u^{-1})^*A,u\cdot\Phi).$
\end{definition}

\begin{remark}
	$\sA(\fs)$ is nonempty and an affine space over $\Omega^1(M,\ad(\fs))$.
\end{remark}

\begin{definition}
	Given a $\Spin^{h}$ structure $\fs$ on $M$ and a spin connection $A\in\sA(\fs) $ the \textbf{Dirac operator} $\slD_{A}: \Gamma(\mathbf S)\to \Gamma(\mathbf S)$ is defined by 
	$$ \slD_{A}\Phi=\sum_{i=1}^{3}\gamma(e_i)\nabla_{A,{e_i}}\Phi$$
	where $ \{ e_1,e_2,e_3\}$ is an oriented local orthonormal frame of $TM$.
\end{definition}

\begin{definition}Given a $\Spin^{h}$ structure $\fs$ on $(M,g)$, the \textbf{$\Sp(1)$-Seiberg-Witten equation} is the following set of equations: for $A\in\sA(\fs) $, $\Phi \in  \Gamma(\bS)$:
	\begin{equation} \label{eq sp(1) SW main}
	\begin{cases*}
		\slD_{A}\Phi=0\\
	F_{\ad(A)}=\mu(\Phi). 
		\end{cases*}\qedhere
		\end{equation}
\end{definition}
\begin{remark} \label{rmk complexification}If we replace $\H$ by the quaternionic hermitian vector space $\H\tn_\C\C^2$ and the quaternionic representation $\rho$ by $\rho^\C:\SU(2)\to \End_\C(\H\tn_\C\C^2)$ defined by $\rho^\C(A)(q\tn v)= q\tn Av$, then will obtain another generalized Seiberg-Witten equation called \textbf{$\SU(2)$-monopole equation}. This is closely related to the $\PU(2)$ monopole equation appeared in the literature. By the following commutative diagram
\[\begin{tikzcd}
\Sp(1) \arrow{r}{\rho^\C} \arrow[swap]{d}{\rho} & \End_\C(\H\tn_\C\C^2)= \End_\C(\C^4)\\
\End_\R(\H) \arrow{ru}{\tn \C}
\end{tikzcd}
\]
we observe that the Dirac operator $\slD^\C_A: \Gamma(\mathbf S\tn_\C\C^2)\to \Gamma(\mathbf S\tn_\C C^2)$ in the $\SU(2)$-monopole equation is the complexification of the Dirac operator $\slD_A:\Gamma(\mathbf S)\to \Gamma(\mathbf S)$ in the $\Sp(1)$-Seiberg-Witten equation (see \cite[Lemma 2.1]{Lim2003}). Moreover, the $\SU(2)$-monopoles $(A,\Phi)$ with $\Phi$ real are exactly the solutions of the $\Sp(1)$-Seiberg-Witten equation.
\end{remark}

\begin{definition}The \textbf{$\Sp(1)$-Seiberg-Witten moduli space} $\sM^h_{SW}$ is defined by
\begin{equation*}
	\sM^h_{SW}:=\frac{\{(A,\Phi)\in \sA(\fs)\times \Gamma(\bS):(A,\Phi) \ \text{satisfies}\ \autoref{eq sp(1) SW main} \}}{\sG(\fs)}.
\end{equation*}
A solution $(A,\Phi)\in \sA(\fs)\times \Gamma(\bS)$ of the equation \autoref{eq sp(1) SW main} is called \textbf{irreducible} if the stabilizer of $(A,\Phi)$ is the trivial group, otherwise it is called \textbf{reducible}.  
\end{definition}
\begin{remark}A solution $(A,\Phi)$ is reducible if and only if $\Phi=0$. Thus the moduli space of reducible solutions is essentially the moduli space of flat $\SU(2)$-connections over $M$.
\end{remark}

	\begin{prop}[{\cite[Proposition 5.1.5]{Morgan1996}}](Lichenerowicz-Weitzenböck formula)
	Suppose $A\in\sA(\fs)$ and $\Phi \in\Gamma(\bS)$. Then
	$$\slD^2_A\Phi=\nabla^*_A\nabla_A\Phi+\tilde\gamma(F_{\ad(A)})\Phi+\frac{\operatorname{scal}_g}{4} \Phi. $$
\end{prop}
\begin{remark}If $(A ,\Phi)$ is a solution of the equation \autoref{eq sp(1) SW main} then 
$$\norm{\nabla_A\Phi}^2_{L^2}+\norm{\mu(\Phi)}^2_{L^2}+\frac 14\int_M\operatorname{scal}_g\abs{\Phi}^2=0.$$
Therefore, if $\operatorname{scal}_g\geq 0$ then $\mu(\Phi)=0$ and hence $\Phi=0$ (see \autoref{prop moment map}). 
\end{remark}
\begin{definition}\label{def sw map linearzi}
    The \textbf{Seiberg-Witten map} $\operatorname{SW}: \sA(\fs)\times \Gamma(\bS)\to \Omega^1(M,\ad(\fs))\times \Gamma(\bS)$ is defined by
    \[\operatorname{SW}(A,\Phi)=(*F_{\ad(A)}-*\mu(\Phi),-\slD_{A}\Phi).\]
Denote by $G_{(A,\Phi)}:\Omega^0(M,\ad(\fs))\to \Omega^1(M,\ad(\fs))\oplus \Gamma(\bS) $, the \textbf{linearization map of the gauge group action} at $(A,\Phi)$, which is given by
\[G_{(A,\Phi)}\xi=(-d_{\ad(A)} \xi,\rho(\xi)\Phi).\]
 	The \textbf{gauge and co-gauge fixed linearization} of the Seiberg-Witten map at a solution $(A,\Phi)$,
	$$\cL_{(A,\Phi)}:\Omega^1(M,\ad(\fs))\oplus \Gamma(\bS)\oplus\Omega^0(M,\ad(\fs))\to \Omega^1(M,\ad(\fs))\oplus \Gamma(\bS)\oplus\Omega^0(M,\ad(\fs))$$ 
	is defined by  
	\[
	\cL_{(A,\Phi)}:=
 \begin{bmatrix}
		d\operatorname{SW}_{|(A,\Phi)}& G_{(A,\Phi)}\\
		G^*_{(A,\Phi)}&0
		\end{bmatrix}=
	\left[ {\begin{array}{ccc}
		*_3d_{\ad(A)}&-2*\mu(\Phi,\cdot)&-d_{\ad(A)}\\
		-\tilde{\gamma}(\cdot)\Phi&-\slD_{A}&\rho(\cdot)\Phi\\
		-d_{\ad(A)}^*&\rho^*(\cdot\Phi^*)&0
		\end{array} } \right].\qedhere
	\]
\end{definition}

\begin{remark}The operator $\cL_{(A,\Phi)}$ is formally self-adjoint and elliptic. Furthermore, the deformation theory of $\sM^h_{SW}$ is controlled by the following elliptic deformation complex:
$$\Omega^0(M,\ad(\fs))\xrightarrow{G_{(A,\Phi)}}\Omega^1(M,\ad(\fs))\oplus \Gamma(\bS)\xrightarrow{d\operatorname{SW}_{|(A,\Phi)}}\Omega^1(M,\ad(\fs))\oplus \Gamma(\bS)\xrightarrow{G^*_{(A,\Phi)}}\Omega^0(M,\ad(\fs)).$$
If $(A ,\Phi)$ is an irreducible solution of the equation \autoref{eq sp(1) SW main} then by \cite[Proposition 3.6, Proposition 2.19]{Doan2017a} the moduli space $\sM^h_{SW}$ around $(A ,\Phi)$ is homeomorphic to the zero set of a smooth map 
\begin{equation*}
\ob:\ker \cL_{(A,\Phi)}\to \coker \cL_{(A,\Phi)}	.\qedhere
\end{equation*}
 \end{remark}
 \begin{definition}The \textbf{Zariski tangent space} of $\sM^h_{SW}$ at an irreducible solution $(A,\Phi)\in \sA(\fs)\times \Gamma(\bS)$ is 
 $$\frac{\ker d\operatorname{SW}_{|(A,\Phi)}}{\im G_{(A,\Phi)} }=\ker d\operatorname{SW}_{|(A,\Phi)}\cap \ker G^*_{(A,\Phi)} =\ker \cL_{(A,\Phi)}.$$
    An irreducible solution $(A,\Phi)\in \sA(\fs)\times \Gamma(\bS)$ of the equation \autoref{eq sp(1) SW main} is called \textbf{infinitesimally rigid} or  \textbf{unobstructed} if $\ker \cL_{(A,\Phi)}=\{0\}$.
\end{definition}
\begin{remark}
Since ${\mu^{-1}(0)}=0$ (see \autoref{prop moment map}), there is a constant $C>0$ such that
$\abs{\Phi}^2\leq C\abs{\mu(\Phi)}$.	 Then for any solution $(A ,\Phi)$ of the equation \autoref{eq sp(1) SW main} we have $\norm{\Phi}_{L^4}$ and $\norm{F_{\ad(A)}}_{L^2}$ are uniformly bounded. Uhlenbeck compactness and elliptic bootstraping \cite[proposition 4.5]{Lim2003} will imply that the moduli space $\sM^h_{SW}$ is compact. As the virtual dimension of the moduli space of irreducible solutions $\sM^{h,*}_{SW}\subset \sM^h_{SW}$ is zero, one might expect to define an topological invariant of $M$ by counting perturbed irreducible solutions and possibly with some correction terms (due to the presence of reducible solutions). This has been carried out by \citet{Lim2003} in the case when $M$ is an integral homology $3$-sphere. The main difficulty for rational homology $3$-spheres or general $3$-manifolds is the presence of more reducible strata, but the we hope that the work of \citet{Bai2020} will be helpful to resolve this issue.  \end{remark}

The following proposition will be useful later to decide if a solution of the equation \autoref{eq sp(1) SW main} is irreducible or not. 
\begin{prop}\label{prop lin square}Let $(A ,\Phi)$ be a solution of the equation \autoref{eq sp(1) SW main}. Then 
	\[\cL^2_{(A,\Phi)}=
	\left[ {\begin{array}{ccc}
		\Delta_{\ad(A)}+\tilde{\gamma}^*(\tilde{\gamma}(\cdot)\Phi\Phi^*)&0&0\\
		0&\slD_{A}^2+\tilde{\gamma}(\tilde{\gamma}^*(\cdot\Phi^*))\Phi+\rho(\rho^*(\cdot\Phi^*))\Phi&0\\
		0&0&\Delta_{\ad(A)}+\rho^*(\rho(\cdot)\Phi\Phi^*)
		\end{array} } \right].
	\]
\end{prop}
In the proof of this proposition we need the following identities.
\begin{lemma}[{\cite[Appendix B]{Doan2017a}}]\label{lem identity DW}
	Suppose $A\in\sA(\fs)$, $\xi\in \Omega^0(M,\ad(\fs))$ and $\phi,\psi\in\Gamma(\bS)$. Then
	 \begin{enumerate}[(i)]
  \item $[\xi,\mu(\phi,\psi)]=\mu(\phi,\rho(\xi)\psi)+\mu(\psi,\rho(\xi)\phi)$,
 \item $d^*_{\ad(A)}\mu(\phi,\psi)=*\mu(\slD_{A}\phi,\psi)+*\mu(\slD_{A}\psi,\phi)-\frac{1}{2}\rho^*((\nabla_A\phi)\psi^*+ (\nabla_A\psi)\phi^*)$.
  \end{enumerate}
	\end{lemma}
\begin{proof}[{Proof of \autoref{prop lin square}}]We have
	\[
	\cL^2_{(A,\Phi)}=
	\left[ {\begin{array}{ccc}
		*_3d_{\ad(A)}&-2*\mu(\Phi,\cdot)&-d_{\ad(A)}\\
		-\tilde{\gamma}(\cdot)\Phi&-\slD_{A}&\rho(\cdot)\Phi\\
		-d_{\ad(A)}^*&\rho^*(\cdot\Phi^*)&0
		\end{array} } \right]
	\left[ {\begin{array}{ccc}
		*_3d_{\ad(A)}&-2*\mu(\Phi,\cdot)&-d_{\ad(A)}\\
		-\tilde{\gamma}(\cdot)\Phi&-\slD_{A}&\rho(\cdot)\Phi\\
		-d_{\ad(A)}^*&\rho^*(\cdot\Phi^*)&0
		\end{array} } \right].
	\]
	Denote by $M_m^n$ the element of $\cL^2_{(A,\Phi)}$ which sits on the $m$-th row and $n$-th column. Since $\cL^2_{(A,\Phi)}$ is formally self adjoint we need to compute only the following to conclude the proposition. 
	\begin{align*}
 &M_1^1= d_{\ad(A)}^*d_{\ad(A)}+2*\mu(\Phi,\tilde{\gamma}(\cdot)\Phi)+d_{\ad(A)}d^*_{\ad(A)}=\Delta_{\ad(A)}+\tilde{\gamma}^*(\tilde{\gamma}(\cdot)\Phi\Phi^*),\\
 &M_2^2=\slD_{A}^2+2\tilde{\gamma}(*\mu(\Phi,\cdot))\Phi+\rho(\rho^*(\cdot\Phi^*))\Phi,\\
 &M_3^3= d_{\ad(A)}^*d_{\ad(A)}+\rho^*(\rho(\cdot)\Phi\Phi^*)=\Delta_{\ad(A)}+\rho^*(\rho(\cdot)\Phi\Phi^*),\\
&M_1^2=-2d_{\ad(A)}^*\mu(\Phi,\cdot)+2*\mu(\Phi,\slD_{A}\cdot)-d_{\ad(A)}\rho^*(\cdot\Phi^*)=-2*\mu(\slD_{A}\Phi,\cdot) (\text{by}\  \autoref{lem identity DW}	)=0,\\
&	M_1^3=-*[F_{\ad(A)},\cdot]-2*\mu(\Phi,\rho(\cdot)\Phi)=-*[\mu(\Phi),\cdot]-2*\mu(\Phi,\rho(\cdot)\Phi) =0  (\text{by}\ \autoref{lem identity DW}	),\\
&	M_2^3=\tilde{\gamma}(d_{\ad(A)}\cdot)-\slD_{A}\rho(\cdot)\Phi=-\rho(\cdot)\slD_{A}\Phi=0\ (\text{as}\ \slD_{A}\Phi=0).\qedhere
	\end{align*}
\end{proof}
\section{Another description of the $\Sp(1)$-Seiberg-Witten equation}\label{section explicit description}
The isomorphism $\Spin^{h}(3)=\frac{\Sp(1)\times \Sp(1)}{\<-1\>}\iso \SO(4)$ can be written as follows:
$$[p_+,p_-]\mapsto \{\Phi\to p_+ \Phi \bar{p}_-\}, \ \ \ \text{where}\ p_\pm\in \Sp(1), \Phi\in \H.$$
With the isomorphism $\Im \H \cong \Lambda^2_{\pm}(\H)$ given by $v\mapsto 1\w v\pm*_3v$ we have the following commutative diagram
\[ \begin{tikzcd}
\frac {\Sp(1)\times \Sp(1)}{\{\pm1\}} \arrow{r}{(\pi_+,\pi_-)} \arrow[swap]{d}{\cong} & \SO(\Im\H)\times \SO(\Im \H) \arrow{d}{\cong} \\%
 \SO(4)\arrow{r}{2:1}&\SO(\Lambda^2_{+}\H)\times \SO(\Lambda^2_{-}\H)\end{tikzcd}
\]
where $\pi_\pm[p_+,p_-]= \{w \to p_\pm w\bar p_\pm\}$. Here the top and bottom maps are $2$-fold coverings, and left and right maps are isomorphisms. And $*_3$ is the Hodge-star operator in dimension $3$.

Let $(M,g)$ be a closed oriented Riemannian $3$-manifold. Set $V := {\R}\oplus T^*M.$ The metric $g$ induces an inner product on the bundle $V$. We choose the \textbf{$\Spin^h$-structure} $\fs=\SO(V)$ with the isomorphism 
  	$\SO(V) \times_{\SO(4)} \SO(3)\cong \SO(TM)$ induced by the above $\pi_+$. Observe that, the \textbf{spinor bundle} is $$\bS= V={\R}\oplus T^*M.$$
   We define $\gamma_\pm: T^*M\to \Lambda^\pm V$ by
$\nu \mapsto 1\wedge \nu \pm *_3\nu.$ More explicitly, for $(f,\sigma)\in \R\oplus T^*M$ and $\nu\in T^*M$
	\begin{equation}\label{ eqn explicit Cliff multi}
	\gamma_\pm(\nu)(f,\sigma)
	= (-\inp{\nu}{\sigma},f\nu  \pm *_3(\nu\wedge\sigma)).
	\end{equation}
Note that, the \textbf{Clifford multiplication} map $\gamma$ is exactly $\gamma_+$, and the adjoint bundle $\ad(\fs)= T^*M$ and  
	 the map $\rho$ is  exactly $-\gamma_-$.

 The space of all spin connections $\sA(\fs)$ is exactly the space of all metric connections on $V$ which induces the Levi-Civita connection $\nabla_{LC}$ on $T^*M$ via $\pi_+$.

\begin{prop}A spin connection $A\in \sA(\fs)$ can be expressed as
\begin{equation}\label{eq explicit A}
	A=
	\begin{bmatrix}   
	d & {a}^*\\
	-a & \nabla_{LC} +*_3(a\wedge\cdot)
	\end{bmatrix}
	\end{equation}
 where $a \in \Omega^1(M,T^*M)$ and $a^* =\inp{a}{\cdot}$.  Moreover, the induced connection on $\ad(\fs)= T^*M$ is $$\ad(A)=\nabla_{LC} +2*_3(a\wedge\cdot).$$
\end{prop}
\begin{proof} 
A metric connection on $V = {\R}\oplus T^*M$ can always be expressed as
	\begin{equation*}
	A=
	\begin{bmatrix}   
	d & {a}^*\\
	-a & \nabla_T
	\end{bmatrix}
	\end{equation*}
where $\nabla_T$ is a metric connection on $T^*M$ and $a \in \Omega^1(M,T^*M)$ with $a^* :=\inp{a}{\cdot}$.
	This connection induces the connection $\nabla_\pm = \nabla_T  \mp *_3(a\wedge \cdot) $
	on $T^*M$ via $\pi_\pm$. Indeed, for $\nu\in \Gamma(T^*M)$,
 \[
	\nabla_A \big(\gamma_\pm(\nu)\big)
	=
	\nabla_A (1 \wedge \nu \pm *_3\nu) 
	=
	1\wedge(\nabla_{T}\nu \mp*_3(a\wedge\nu))
	\pm *_3(\nabla_T\nu \mp*(a\wedge\nu)) =
	\gamma_\pm(\nabla_{T}\nu \mp*_3(a\wedge\nu)).
	\]
Here, we have used the identities,
$\nabla_A (1 \wedge \nu)=-a \wedge \nu+ 1 \wedge \nabla_{T}\nu$, and 
	\[\nabla_A (*_3\nu)=1 \wedge \iota(a^* (*_3\nu) + *_3\nabla_T\nu 
	=-1 \wedge *_3(a\wedge\nu) + *_3\nabla_T\nu.
	\]
Thus $\nabla_{LC}=\nabla_+$ if and only if
	$\nabla_T = \nabla_{LC} +*_3(a\wedge\nu)$. And, $\ad(A)=\nabla_-$ if and only if
	$\nabla_T = \nabla_{LC} +2*_3(a\wedge\nu)$.
\end{proof}
\begin{definition}
    For $a\in T^*M\tn T^*M$ we write $a = a_{ij}e^i\otimes e^j$ in an oriented local orthonormal frame $ \{ e_1,e_2,e_3\}$ of $TM$ and we define
\begin{equation*}
	\tr(a) := \sum_{i=1}^3 \inp{a(e_i)}{e_i}, \quad
	\tau(a):= *_3\sum_{i=1}^3 e_i \wedge a(e_i)\in T^*M,
	\end{equation*}
 and \[S(a) := \sum_{i,j=1}^3 (a_{ij}+a_{ji}) e^i\otimes e^j\in \sym^2(T^*M).\qedhere\]
 \end{definition}
\begin{prop}For a spin connection $A\in \sA(\fs)$ with decomposition in \autoref{eq explicit A},
	the \textbf{Dirac operator} $\slD_{A}$ can be expressed as
	\begin{equation*}
	\slD_{A}(f,\sigma)
	=
	\begin{bmatrix}
	d^*\sigma - \inp{\tau(a)}{\sigma} + \tr(a)f \\
	df  + *_3d\sigma-f\tau(a) - \tr(a)\sigma +  \iota(\sigma)S(a)
	\end{bmatrix}.
	\end{equation*}
\end{prop}

\begin{proof}First, we see that
	\begin{align*}
	\nabla_{A,e_i}(f,\sigma)
	=
	\begin{bmatrix}
	\partial_i f+\iota(a(e_i)^*)\sigma \\
	- a(e_i)f + \nabla_{LC,e_i}\sigma  + *_3(a(e_i)\wedge\sigma)
	\end{bmatrix}.
	\end{align*}
Therefore
	\begin{align*}
	\slD_{A}(f,\sigma)
	&=\sum_{i=1}^3 \gamma(e_i)\nabla_{A,e_i}(f,\sigma) \\
	&=
	\sum_{i=1}^3 \gamma_+(e_i)
	\begin{bmatrix}
	\partial_i f + i(a(e_i)^*)\sigma \\
	-a(e_i)f + \nabla_{LC,e_i}\sigma  +  *_3(a(e_i)\wedge\sigma)
	\end{bmatrix} \\
	&=\renewcommand\arraystretch{2}
	\begin{bmatrix}
	d^*\sigma + \sum_{i=1}^3 (\inp{a(e_i)}{e_i}f  -\inp{*_3(a(e_i)\wedge\sigma)}{e_i} \\
	d f + *_3d\sigma + \sum_{i=1}^3 ( e^i \iota(a(e_i))\sigma -*_3(e^i\wedge a(e_i)) f +*_3(e^i\wedge *_3(a(e_i)\wedge\sigma)))
	\end{bmatrix} \\
	&=\renewcommand\arraystretch{2}
	\begin{bmatrix}
	d^*\sigma -\inp{\tau(a)}{\sigma} +\tr(a)f \\
	d f   + *_3d\sigma-f\tau(a) + \underbrace{\sum_{i=1}^3  (e^i \iota(a(e_i))\sigma +*_3(e^i\wedge *_3(a(e_i)\wedge\sigma)))}_{=:B(a,\sigma)}
	\end{bmatrix} .     
	\end{align*}
To see $B(a,\sigma)= -\tr(a)\sigma + i(\sigma)S(a)$ we do a direct computation.	
	\begin{align*}
	B(a,\sigma)&=\sum_{i,j}
	a_{ij}\sigma_je^i
	+\sum_{k}a_{ij}\sigma_k *_3(e^i\wedge*_3(e^j\wedge e^k)) \\
	&=
	\sum_{i,j}a_{ij}\sigma_je^i
	+\sum_{k,\ell,m}a_{ij}\sigma_k\epsilon_{jk\ell}\epsilon_{i\ell m} e^m \\
	&=
	\sum_{i,j}a_{ij}\sigma_je^i
	-\sum_{k,m} a_{ij}\sigma_k(\delta_{ji}\delta_{km}-\delta_{jm}\delta_{ki}) e^m =
	\sum_{i,j}a_{ij}\sigma_je^i
	-\tr(a) \sigma + \sum_{i,j}a_{ij}\sigma_ie^j. \qedhere
	\end{align*}
\end{proof}

\begin{prop}\label{prop moment map} The \textbf{moment map} $\mu:\R\oplus T^*M\to \Lambda^2(T^*M)\tn T^*M\cong T^*M\otimes T^*M$ can be expressed as
	
\begin{equation*}
	\mu(f,\sigma) = (f^2-\abs{\sigma}^2)g -2 *_3(f\sigma) +  \sigma\otimes\sigma.
	\end{equation*}
\end{prop}

\begin{proof}
	We have
	\begin{align*}
	2\inp{\mu(f,\sigma)}{\nu\otimes\xi}
	&=
	-\inp{(f,\sigma)}{\gamma_+(\nu)\gamma_-(\xi)(f,\sigma)} \\
	&=
	\inp{\gamma_+(\nu)(f,\sigma)}{\gamma_-(\xi)(f,\sigma)}\\
	&=
	\inp{(-\inp{\nu}{\sigma},
		f\nu  + *_3(\nu\wedge\sigma))}{(-\inp{\xi}{\sigma},
		f\xi - *_3(\xi\wedge\sigma))}\\
  &=2\inp{\sigma \otimes \sigma}{\nu\otimes \xi}
	+ f^2\inp{\nu}{\xi}
	- 2\inp{f*_3\sigma}{\nu\wedge\xi}
	- \abs{\sigma}^2\inp{\nu}{\xi}\\
 &=2\inp{\sigma \otimes \sigma}{\nu\otimes \xi}
	+ 2f^2\inp{g}{\nu\otimes\xi}
	- 4\inp{f*_3\sigma}{\nu\otimes\xi}
	- 2\abs{\sigma}^2\inp{g}{\nu\otimes\xi}.\qedhere
	\end{align*}
\end{proof}

On the adjoint bundle $\ad(\fs)= T^*M$ the Lie bracket, $[v,w]=2*_3v\w w$ for $v,w \in T^*M$. Therefore, $\ad(A)=\nabla_{LC} +2*_3(a\wedge\cdot)=\nabla_{LC}+[a,\cdot].$  Then the curvature $F_{\ad(A)}\in \Omega^2(M,T^*M) $ can be expressed as 
	$$F_{\ad(A)}= R_g+d_{LC}a+ \frac{1}{2}[a\wedge a]$$
where $R_g\in \Omega^2(M,\Lambda^2T^*M)\cong \Omega^2(M,T^*M)\cong \Omega^1(M,T^*M) $ is the Riemann curvature of $g$.

The \textbf{$\Sp(1)$-Seiberg-Witten equation} \autoref{eq sp(1) SW main} can be rephrased as follows:  for $a\in  \Omega^1(M,T^*M)$, $f\in \Omega^0(M,\R)$ and $ \sigma\in \Omega^1(M,\R)$
\begin{equation}\label{eq: explicit SW}
\begin{cases}
d^*\sigma - \inp{\tau(a)}{\sigma}+\tr(a)f=0\\
df +*_3d\sigma - f \tau(a)  -\tr(a)\sigma +\iota(\sigma)S(a)=0\\
R_g +*_3d_{LC}a+\frac{1}{2}*_3[ a\wedge a]= f^2g-\abs{\sigma}^2g-2*_3(f\sigma)+\sigma\tn \sigma.
\end{cases}
\end{equation}

The Seiberg-Witten map  (see \autoref{def sw map linearzi}) $\operatorname{SW}:\Omega^1(M,T^*M) \oplus  \Omega^0(M,\R)\oplus \Omega^1(M,\R)\to \Omega^1(M,T^*M) \oplus  \Omega^0(M,\R)\oplus \Omega^1(M,\R)$ can be expressed as
\begin{align*}
\operatorname{SW}(a,f,\sigma)=
(R_g +*_3d_{LC}a+ \frac{1}{2}*_3[ a\wedge a]- \mu(f,\sigma),-d^*\sigma +\inp{\tau(a)}{\sigma}-\tr(a)f,\\
-df   -*_3d\sigma +f \tau(a)+\tr(a)\sigma -i(\sigma)S(a)).
\end{align*} 
A direct computation yields the linearization of $\operatorname{SW}$ at $(a,f,\sigma)$, which is
	\[
	 {d\operatorname{SW}}_{(a,f,\sigma)}=
	\left[ {\begin{array}{ccc}
		*_3d_{LC}+*_3[a\w\cdot] &  -2fg+2*_3\sigma&  2\inp{\sigma}{\cdot}g+2f*_3-S(\sigma\tn \cdot) \\
		\inp{\tau(\cdot)}{\sigma}- \tr(\cdot)f & -\tr(a) & -d^*+\inp{\tau(a)}{\cdot} \\
		f\tau(\cdot)+\tr(\cdot)\sigma-i(\sigma)S(\cdot) & -d+\tau(a) & -*_3d+\tr(a)-i(\cdot)S(a)
		\end{array} } \right].
	\]

The linearization of the gauge group action at $(a,f,\sigma)\in \Omega^1(M,T^*M) \oplus  \Omega^0(M,\R)\oplus \Omega^1(M,\R)$ (see \autoref{def sw map linearzi}),
$G_{(a,f,\sigma)}:\Omega^0(M,T^*M)\longrightarrow \Omega^1(M,T^*M) \oplus  \Omega^0(M,\R)\oplus \Omega^1(M,\R)$ is given by
$$G_{(a,f,\sigma)}=(-d_{LC}-[a\w\cdot],\inp{\sigma}{\cdot},-f-*_3(\cdot\w \sigma)).$$

Finally, the \textbf{gauge and co-gauge fixed linearization} of the Seiberg-Witten map  $\operatorname{SW}$ at a solution $(a,f,\sigma)$ of \autoref{eq: explicit SW}, $\mathcal L_{(a,f,\sigma)}:\Omega^1(M,T^*M) \oplus  \Omega^0(M,\R)\oplus \Omega^1(M,\R))\oplus \Omega^0(M,T^*M)\to 
	\Omega^1(M,T^*M) \oplus  \Omega^0(M,\R)\oplus \Omega^1(M,\R))\oplus \Omega^0(M,T^*M)$
is
\[
	\mathcal L_{(a,f,\sigma)}=
	\left[ {\begin{array}{cc}
		{d\operatorname{SW}}_{(a,f,\sigma)} & G_{(a,f,\sigma)}  \\
		{ G}^*_{(a,f,\sigma)}& 0 
		\end{array} } \right]
  \]
  \[
 = \left[ {\begin{array}{cccc}
		*_3d_{LC}+*_3[a\w\cdot] &  -2fg+2*_3\sigma&  2\inp{\sigma}{\cdot}g+2f*_3-S(\sigma\tn \cdot)& -d_{LC}-[a\w\cdot] \\
		\inp{\tau(\cdot)}{\sigma}- \tr(\cdot)f & -\tr(a) & -d^*+\inp{\tau(a)}{\cdot}&\inp{\sigma}{\cdot} \\
		f\tau(\cdot)+\tr(\cdot)\sigma-i(\sigma)S(\cdot) & -d+\tau(a) & -*_3d+\tr(a)-i(\cdot)S(a)&-f-*_3(\cdot\w \sigma)\\
  -d^*_{LC}-2(*a\w\cdot)^*&\sigma&-f+*_3(\cdot\w \sigma)&0
		\end{array} } \right].
	\]

\section{$\Sp(1)$-Seiberg-Witten equation on hyperbolic $3$-manifold}
If we force $a = 0\in \Omega^1(M,T^*M)$ in the $\Sp(1)$ Seiberg--Witten equation \autoref{eq: explicit SW}  then it becomes
\begin{equation}\label{eq: explicit SW a=0}
\begin{cases}
d^*\sigma = 0 \\
d f + *_3d\sigma = 0,\\
R_g = (f^2-\abs{\sigma}^2)g - 2*_3(f\sigma) + \sigma\otimes\sigma.
\end{cases}
\end{equation}
\subsection{An irreducible solution: $(a,f,\sigma)=(0, \pm 1,0)$}\label{sub sec irred sol}
Suppose $(M,g)$ is an oriented closed \textbf{hyperbolic 3 manifold} of constant sectional curvature $-1$. Then $R_g=g\in \Omega^1(M,T^*M)$. This implies that
$(a,f,\sigma)=(0, \pm 1,0)\in \Omega^1(M,T^*M) \oplus  \Omega^0(M,\R)\oplus \Omega^1(M,\R)$ are two gauge equivalent irreducible solutions of the $\Sp(1)$ Seiberg--Witten equation \autoref{eq: explicit SW}. We will work below with one of them say, $(0, 1,0)$. We have the following proposition about the linearization map at this solution, which essentially says when this solution is infinitesimally rigid.
\begin{definition} A symmetric $(0,2)$-tensor $a\in \Omega^1(M,T^*M)$ is called \textbf{Codazzi tensor} if 
\[d_{LC}a=0\in \Omega^2(M,T^*M).\qedhere\]
    \end{definition}
\begin{prop}\label{prop linsquare at special sol} The square of the linearlization
    \[
    \mathcal L_{(0,1,0)}^2=\left[ {\begin{array}{cccc}
	\Delta_{LC}+2g \tr(\cdot)+2*_3\tau(\cdot)& 0 & 0&0\\
	0& \Delta+6 &0&0\\
	0&0 & \Delta+5&0\\
	0&0&0&\Delta_{LC}+1
	\end{array} } \right].
\]
Moreover, \begin{align*}
    \ker \mathcal L_{(0,1,0)}&\cong\{a\in \sym^2(T^*M): d_{LC}a=0, \ \tr(a)=0 \}=\{\text{Trace-free Codazzi tensors on $M$}\}.
\end{align*}
\end{prop}
\begin{proof}Since $a=0,f=1,\sigma=0$, therefore from the description in \autoref{section explicit description} of the linearization we have
\[
\mathcal L^2_{(0,1,0)}=
\left[ {\begin{array}{cccc}
	*_3d_{LC} & - 2g &  2*_3& -d_{LC}\\
	-\tr(\cdot) & 0 & -d^*&0\\
	\tau(\cdot) & -d & -*_3d&-1\\
	-d^*_{LC}&0&-1&0
	\end{array} } \right]
 \left[ {\begin{array}{cccc}
	*_3d_{LC} & - 2g &  2*_3& -d_{LC}\\
	-\tr(\cdot) & 0 & -d^*&0\\
	\tau(\cdot) & -d & -*_3d&-1\\
	-d^*_{LC}&0&-1&0
	\end{array} } \right].
\]
By \autoref{prop lin square}, we obtain that all the off-diagonal terms of $\mathcal L^2_{(0,1,0)}$ are $0$. Therefore $\mathcal L_{(0,1,0)}^2$ is
 \[
\left[ {\begin{array}{cccc}
	d_{LC}^*d_{LC}+2g \tr(\cdot)+2*_3\tau(\cdot)+d_{LC}d_{LC}^*& 0 & 0&0\\
	0& 2\tr(g)+d^*d &0&0\\
	0&0 & 2\tau *_3+dd^*+d^*d+1&0\\
	0&0&0&d_{LC}^*d_{LC}+1
	\end{array} } \right].
\]
Since $\tau*_3\sigma=2\sigma$ for all $\sigma\in T^*M$ we obtain the required form of $\mathcal L_{(0,1,0)}^2$.

Thus $(a,f,\sigma,\xi)\in \ker \mathcal L_{(0,1,0)}=\ker \mathcal L_{(0,1,0)}^2$ if and only if 
\begin{equation}
  f=0,\quad \sigma=0,\quad \xi=0,\quad \Delta_{LC}a+2g \tr(a)+2*_3\tau(a)=0. 
\end{equation}
Since trace commutes with $\Delta_{LC}$ and $\tr(*_3\tau(a))=0$ therefore 
$\Delta \big(\tr (a)\big)+6 \tr (a)=0$. Hence $\tr(a)=0$.  Again, $\tau$ commutes with $\Delta_{LC}$ and $\tau(*_3\tau(a))=2\tau(a)$ and therefore 
$\Delta \big(\tau(a)\big)+4 \tau (a)=0$. Hence $\tau(a)=0$ as well; in particular, by definition, $a$ is a symmetric tensor. Thus $a$ is a trace free symmetric tensor which is harmonic i.e. $\Delta_{LC}a=0$. Since $\Delta_{LC}=d_{LC}^*d_{LC}+d_{LC}d^*_{LC}$,
$$0=\inp{\Delta_{LC}a}{a}_{L^2}=\norm{d_{LC}a}_{L^2}^2+\norm{d^*_{LC}a}_{L^2}^2.$$
This implies that $d_{LC}a=0$ and $d^*_{LC}a=0$.
Therefore, \begin{align*}
    \ker \mathcal L_{(0,1,0)}&\cong\{a\in \sym^2(T^*M): d_{LC}a=0,d^*_{LC}a=0,  \tr(a)=0 \}.
\end{align*}
In fact, here $d^*_{LC}a=0$ is redundant as it follows from other conditions. Indeed, for a symmetric $(0,2)$-tensor $a\in \Omega^1(M,T^*M)$,  
$$d^*_{LC}a=-\nabla_{LC}(\tr (a))+\tr_{2,3}(d_{LC}a).$$
cf.~\cite[Proof of Proposition 9.4.4]{Petersen2016}. 
\end{proof}

\begin{cor} The irreducible solution $\big(\nabla_{LC},(1,0)\big)$ of \autoref{eq: explicit SW} is infinitesimally rigid (or, unobstructed) if and only if the hyperbolic $3$-manifold $(M,g)$ does not admit any trace-free Codazzi tensors.
    \end{cor}
   We need the following lemma from the literature which provides a sufficient condition for $\big(\nabla_{LC},(1,0)\big)$ to be infinitesimally rigid.
 \begin{lemma}[{\citet[Lemma 6]{Lafontaine1983}}]\label{thm Lafontaine}
	Let $(M,g)$ be a closed hyperbolic $3$-manifold ${\H}^3/\Gamma$ with $\Gamma$ being a co-compact  discrete subgroup of $\SO^+(1,3)$ and $g$ is the hyperbolic metric. Then there is an injection
 \[\{a\in \sym^2_0(M,g) :d_{LC} a=0\}=\{\text{Trace-free Codazzi tensors}\} \hookrightarrow H^1(\Gamma,\R^{1,3})\]
 where $H^1(\Gamma,\R^{1,3})$ is the first group cohomology of the $\Gamma$-module $\R^{1,3}$.
\end{lemma}  
\subsection{Locally conformally flat structures and Codazzi tensors}
We review the basics of locally conformally flat structures and its relation with Codazzi tensors. For more detailed discussions we refer the reader to \cites{moroianu2015cotton,Beig1997,Goldschmidt1984}. An oriented closed Riemannian 3-manifold $(M,g)$ is called \textbf{locally conformally flat} if for each point $x\in M$ there exists a open neighbourhood $U_x$ of $x$ and $f\in C^\infty(U_x)$ such that $e^{2f}g$ is flat. The \textbf{Schouten tensor} $P_g\in \Omega^1(M,T^*M)$ and the \textbf{Cotton tensor} $C_g\in \Omega^1(M,T^*M)$ of $g$ are respectively 
$$P_g=\operatorname{Ric}_g-\frac{\operatorname{scal}_g}4 g, \quad C_g=*_3d_{LC}P_g.$$
It is a standard fact that the $3$-manifold $(M,g)$ is locally conformally flat if and only if the Cotton tensor $C_g=0$. There is a \textbf{Chern-Simons functional} $ {CS}:\fM:\to\R$ defined by
		$${CS}(g)=-\frac1{16\pi^2}\int_M\text{tr}(\omega\w d\omega+\frac23 \omega\w\omega\w\omega)$$
		where $\fM$ is the \textbf{space of Riemannian metrics} on $M$ and $\omega$ is the Levi-Civita connection $1$-form with respect to a global orthonormal frame on $(M,g)$. Furthermore, its linearization at $g$  is 
$$d{CS}_{|g}(h)=-\frac1{8\pi^2}\int_M\inp{h}{C_g}_g \operatorname{vol}_g.$$
In fact this implies that the Cotton tensor $C_g$ is symmetric, trace-free, divergence free and conformally covariant (in the sense $e^{-f} {C}_g= C_{e^{2f}g} \ \ \forall f\in C^\infty(M)$). We can consider the map $C:\fM\to \Omega^1(M,T^*M)$, $g\mapsto C_g$.
The \textbf{moduli space of locally conformally flat structures} is then $C^{-1}(0)/ {\Diff(M)\times C^\infty(M)}$. The deformation theory of this moduli space at a locally conformally flat structure $[g]$ is controlled by the following formally self-adjoint conformally invariant elliptic deformation complex: 
\begin{equation}\label{eq def comp of cotton}
    0\to \Omega^0(M,TM)\xrightarrow{ L^0} \sym^2_0(M,g)\xrightarrow{d{C}_{|g}} \sym^2_0(M,g)\xrightarrow{d_{LC} ^*}\Omega^0(M,TM)\to0
\end{equation}
where 
		$L^0(X)= \cL_Xg-\frac23 \operatorname{div}(X)g$ is the linearization of the action of $\Diff(M)$ and $\sym^2_0(M,g)$ is the space of all symmetric trace-free $(0,2)$-tensors on $(M,g)$. The cohomologies $H^0(C,g):=\operatorname{ker}( L^0)$ is the space of all conformal Killing vector fields, 
  $$H^1(C,g):=\frac{\ker(d{C}_{|g})}{\im(L^0)}=\ker(d{C}_{|g})\cap\operatorname{ker}(d_{LC} ^*)$$ is the \textbf{Zariski tangent space} of $\fM_{lcf}$, the moduli space of locally conformally flat structures at $[g]$. We say $g$ is \textbf{infinitesimally rigid} if $H^1(C,g)=0$. If $M$ is simply connected then $H^1(C,g)=0$.

To complete the proof of \autoref{thm hyperbolic sol} we need the following lemma again from the literature.
\begin{lemma}[{\citet[Section 4]{Beig1997}}]
	Let $(M,g)$ be a closed hyperbolic $3$-manifold ${\H}^3/\Gamma$ with $\Gamma$ being a co-compact  discrete subgroup of $\SO^+(1,3)$ and $g$ is the hyperbolic metric. Then $g$ is locally conformally flat and
	\[H^1(C,g)=\{a\in \sym^2_0(M,g) :d_{LC} a=0\}=\{\text{Trace-free Codazzi tensors}\}.\]
\end{lemma}
\begin{cor}\label{cor zariski lcf} The Zariski tangent space of $\sM^h_{SW}$ at $\big(\nabla_{LC},(1,0)\big)$ is same as the Zariski tangent space of $\fM_{lcf}$ at $[g]$. In particular, $\big(\nabla_{LC},(1,0)\big)$ is infinitesimally rigid if and only if $g$ is infinitesimally rigid as a locally conformally flat structure.
    \end{cor}
\begin{proof}[{\large{\normalfont{\textbf{Proof of \autoref{thm hyperbolic sol}}}}}] The hyperbolic metric $g$ induces the irreducible solution $\big(\nabla_{LC},(1,0)\big)$ of \autoref{eq: explicit SW}, which has been shown in \autoref{sub sec irred sol}. Here $A_0=\nabla_{LC}$ is the Levi-Civita connection on $\R\oplus T^*M$ and $\Phi_0=(1,0)\in \Gamma(\R\oplus T^*M)$ is the spinor. The Zariski tangent space of $\sM^h_{SW}$ at $(A_0,\Phi_0)$ is same as the Zariski tangent space $H^1(C,g)$ of $\fM_{lcf}$ at $[g]$ is the \autoref{cor zariski lcf}. The Zariski tangent space of $\sM^h_{SW}$ at $(A_0,\Phi_0)$ is $\ker \cL_{(A_0,\Phi_0)}$ and \autoref{prop linsquare at special sol} proves that it is same as the space of trace-free Codazzi tensors. This completes the proof of the theorem.
\end{proof}
\begin{example}\label{eg rigid}A sufficient condition for $\big(\nabla_{LC},(1,0)\big)$ to be infinitesimally rigid is $H^1(\Gamma,\R^{1,3})=0$. There are examples in the literature (see \cite[Theorem 2]{Kapovich1994}, \cite[Theorem 1.1]{Francaviglia2008}, \cite[Section 5]{Scannell2002}) of infinitely many hyperbolic $3$-manifolds which are obtained by Dehn surgery on hyperbolic $2$-bridge knots or some generalizations and having $H^1(\Gamma,\R^{1,3})=0$. Moreover, in the Hodgson–Weeks census, out of the first $4500$ two generator hyperbolic $3$- manifolds $4439$ are having $H^1(\Gamma,\R^{1,3})=0$ (see \cite[Section 5]{Cooper2006}). 
    \end{example}

\section{$\Sp(1)$-Seiberg-Witten equation over circle times Riemann surface}
In this section we consider $\Sp(1)$-Seiberg-Witten equation over $M= S^1 \times \Sigma$, where $\Sigma$ is a closed Riemann Surface. Fix a Riemannian metric $g_{\Sigma}$ on $\Sigma$ and consider product metric on $S^1\times \Sigma$. Therefore $V={\R}\oplus{\R}\oplus T^*\Sigma$. To prove \autoref{thm SW Riemann surface} we are going to use the following standard lemma.
\begin{lemma}[{\citet[Theorem 3.8]{Doan19multi}}]\label{thm Doan} 
If the $\Sp(1)$-Seiberg--Witten equation \autoref{eq sp(1) SW main} over $S^1 \times \Sigma$ admits an irreducible solution, then every solution is gauge equivalent to a circle-invariant solution with the connection in temporal gauge; in particular, it is pulled back from $\Sigma$, i.e. a  circle invariant configuration  in the sense of \cite[Definition 3.3]{Doan19multi}.
\end{lemma}
  \begin{proof}[{\large{\normalfont{\textbf{Proof of \autoref{thm SW Riemann surface}}}}}] By \autoref{thm Doan}, any irreducible solution $(a,f,\sigma)$ of \autoref{eq: explicit SW} is gauge equivalent to a circle-invariant solution with the connection in temporal gauge. In other words, we may assume that
  \begin{itemize}
  \item $a=\beta\tn dt+\delta, \ \text{where}\ \beta\in \Omega^1(\Sigma,\R), \delta\in \Omega^1(\Sigma,T^*\Sigma)$,
  \item $f\in \Omega^0(\Sigma,\R)$, and $\sigma=\lambda dt+\omega$ where $\lambda\in \Omega^0(\Sigma,\R)$ and $\omega \in \Omega^1(\Sigma,\R)$.
  \end{itemize}
 We are going to use only the last equation of \autoref{eq: explicit SW}:
\begin{equation}\label{eq expli curvature moment}
  R_g +*_3d_{LC}a+\frac{1}{2}*_3[ a\wedge a]= f^2g-\abs{\sigma}^2g-2*_3(f\sigma)+\sigma\tn \sigma.  
\end{equation}
We introduce a notation where we write an element $B\in \Omega^1(S^1\times \Sigma,\R\oplus \R\oplus T^*\Sigma)$ with the decomposition is $B=B_{11}dt\tn dt+B_{12}\tn dt+dt\tn B_{21}+B_{22}$ as a matrix
\[B=
\begin{bmatrix}
   B_{11} &B_{12}\\ 
   B_{21} &B_{22} 
\end{bmatrix}.
\]
With this notation, a direct computation shows us the following:
\[(f^2-\abs{\sigma}^2)g=
\begin{bmatrix}
   f^2-\lambda^2-\abs{\omega}^2&0\\
   0& (f^2-\lambda^2-\abs{\omega}^2)g_\Sigma
\end{bmatrix},\quad 
-2*_3f\sigma=
\begin{bmatrix}
  0&-2f*_\Sigma \omega\\
   2f*_\Sigma \omega& -2f\lambda\vol_\Sigma
\end{bmatrix},
\]
\[\sigma\tn\sigma=
\begin{bmatrix}
  \lambda^2& \lambda \omega\\
   \lambda \omega& \omega\tn \omega
\end{bmatrix},\quad
R_g=
\begin{bmatrix}
 R_{g_\Sigma}& 0\\
   0&0
\end{bmatrix},\quad 
*_3d_{LC}a=
\begin{bmatrix}
 *_\Sigma d\beta& 0\\
   *_\Sigma d_{LC}\delta& 0
\end{bmatrix}
\]
and
\[\frac 12 *_3[a\w a]=
\begin{bmatrix}
 {\langle\delta \w\delta\rangle}& 0\\
   2*_\Sigma(\beta\w *_\Sigma\delta)&0
\end{bmatrix}.
\]
Thus \autoref{eq expli curvature moment} is equivalent to
\begin{equation}\label{eq2 expli curvature moment}
\begin{cases*}
     R_{g_\Sigma}+ *_\Sigma d\beta+ {\langle\delta \w\delta\rangle}= f^2-\abs{\omega}^2\\
     *_\Sigma d_{LC}\delta+2*_\Sigma(\beta\w *_\Sigma\delta)=\lambda\omega+2f*_\Sigma\omega\\
   \lambda\omega-2f*_\Sigma\omega=0\\  
    (f^2-\lambda^2-\abs{\omega}^2)g_\Sigma-2f\lambda\vol_\Sigma+\omega\tn \omega=0.
\end{cases*}
  \end{equation}
But the last equation of \autoref{eq2 expli curvature moment} implies that $$\omega\tn \omega-\frac 12 \abs{\omega}^2g_\Sigma=0,\quad f^2-\lambda^2-\frac 12\abs{\omega}^2=0, \quad f\lambda=0.$$
Thus $\omega=0$ and therefore $f=0$ and $\lambda=0$ as well. Hence the $\Sp(1)$-Seiberg-Witten equation \autoref{eq sp(1) SW main} over $S^1 \times \Sigma$ does not admit any irreducible solution. Hence the only solutions are reducible solutions which are spin connections $A\in \sA(\fs)$ satisfying $F_{\ad(A)}=0$, which are same as flat $\SU(2)$ connections over $M$. This completes the proof.
\end{proof} 

\begin{remark} In general, we can express the solutions of the $\Sp(1)$-Seiberg-Witten equation \autoref{eq sp(1) SW main} as solutions of a vortex equation corresponding to the $\SU(2)$-monopole equation discussed in \autoref{rmk complexification}. We can choose a $\Spin^h$-structure on $S^1 \times \Sigma$ such that the complexification of the spinor bundle is the pullback of $E\oplus (K^{-1}_\Sigma\tn E)$ for some $\U(2)$-bundle $E$ over $\Sigma$ with $\det E= K_\Sigma$ (see \cites[Proposition 4.1]{Okonek1996}[Theorem 44]{Echeverria2021}). One such choice can be $E=\C\oplus K_\Sigma$. Denote by $\sA_c(E)$ the space unitary connections on $E$ inducing the Chern connection on $\det E$.  Moreover, the solutions of the $\SU(2)$-monopole equation are gauge equivalent to either the solutions $(A,\psi_1)\in \sA_c(E)\times \Gamma(E)$ of the vortex equation
\begin{equation}\label{eq vortex 1}
\begin{cases*}
    \bar{\partial}_A\psi_1=0\\
    i*_\Sigma F^0_A+(\psi_1\psi_1^*)_0=0,
\end{cases*}
    \end{equation}
  or, the solutions $(A,\psi_2)\in \sA_c(E)\times \Gamma( K^{-1}_\Sigma\tn E)$ of the vortex equation
\begin{equation}\label{eq vortex 2}
\begin{cases*}
    \bar{\partial}^*_A\psi_2=0\\
    i F^0_A-(\psi_2\psi_2^*)_0=0.
\end{cases*}
    \end{equation}
By Serre duality, \autoref{eq vortex 2} can be identified with \autoref{eq vortex 1} with $E$ replaced by $K^{-1}_\Sigma\tn E$. To satisfy \autoref{eq sp(1) SW main}, $\psi_1$ and $\psi_2$ have to be real and in that case $\psi_1$ and $\psi_2$ are locally constants. Using \autoref{thm SW Riemann surface} we actually conclude that $\psi_1=0$ and $\psi_2=0$. 
\end{remark}
\printbibliography
\end{document}